\documentclass{amsart}
\usepackage{amscd}
\usepackage{amssymb}
\usepackage{graphicx}

\theoremstyle{plain}
\newtheorem{theorem}{Theorem}
\newtheorem{claim}{Claim}

\newtheorem{lemma}{Lemma}
\newtheorem{fact}{Fact}
\newtheorem{corollary}{Corollary}

\theoremstyle{definition}
\newtheorem{definition}{Definition}

\newtheorem{notation}{Notation}

\newcommand{\reduced}{\operatorname{reduced}}

\newcommand{\rii}{{\rm{RI\!I}}}
\newcommand{\riii}{{\rm{RI\!I\!I}}}
\newcommand{\iii}{\operatorname{I\!I\!I}}
\newcommand{\ri}{{\rm{RI}}}

\begin{document}
\title[New deformations on spherical curves and \"{O}stlund Conjecture]{New deformations on spherical curves and \"{O}stlund Conjecture}
\author{Megumi Hashizume}
\address{4-21-1 Nakano, Nakano-ku, Tokyo, Japan 164-8525, Meiji University Organization for the Strategic Coordination of Research
 and Intellectual Properties}
\address{Current address: Takabatakechyo ,Nara-shi, Nara-ken, Japan, 8528, Nara University of Education Center for Educational Research of Science and Mathematics}
\email{hashizume.megumi.y9@nara-edu.ac.jp}
\author{Noboru Ito}
\address{Graduate School of Mathematical Sciences, The University of Tokyo, 3-8-1, Komaba, Meguro-ku, Tokyo 153-8914, Japan}
\email{noboru@ms.u-tokyo.ac.jp}
\date{\today}
\maketitle

\begin{abstract}
In \cite{FHIKM}, a deformation of spherical curves called deformation type $\alpha$ was introduced.
Then, it was showed that 
if two spherical curves $P$ and $P'$ are equivalent under the relation consisting of deformations of type $\ri$ and type $\riii$ up to ambient isotopy, and satisfy certain  conditions, then $P'$ is obtained from $P$ by a finite sequence of deformations of type $\alpha$.
In this paper, we introduce a new type of deformations of spherical curves, called deformation of type $\beta$.
The main result of this paper is: Two spherical curves $P$ and $P'$ are equivalent under (possibly empty) deformations of type $\ri$ and a single deformation of type $\riii$ up to ambient isotopy if and only if $\reduced(P)$ and $\reduced(P')$ are transformed each other by exactly one deformation which is of type $\riii$, type $\alpha$, or type $\beta$ up to ambient isotopy, where $\reduced(Q)$ is the spherical curve which does not contain a $1$-gon obtained from a spherical curve $Q$ by applying deformations of type $\ri$ up to ambient isotopy.
\end{abstract}

\section{Introduction}\label{intro}
A spherical curve is the image of a generic immersion of a circle into a $2$-sphere.
Every spherical curve is transformed into the simple closed curve by a finite sequence of deformations, each of which is either one of type $\ri$, type $\rii$, or type $\riii$ that is a replacement of a part of the spherical curve contained in a disk as shown in Figure~\ref{reidemei} and ambient isotopies.  The deformations of type $\ri$, type $\rii$, and type $\riii$ are obtained from Reidemeister moves of type $\Omega_1$, type $\Omega_2$, and type $\Omega_3$ on knot diagrams by ignoring over/under informations near the crossing points.
We note that each of the deformations means a replacement of a part of a immersed circle fixing the ambient 2-sphere.

\begin{figure}[h!]
\includegraphics[width=10cm]{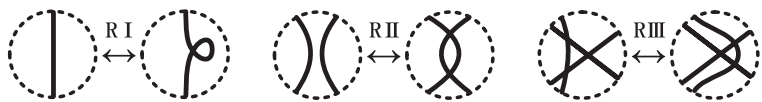}
\caption{Deformations of type $\ri$, type $\rii$, and type $\riii$}\label{reidemei}
\end{figure}

In 2001, from viewpoints of singularity theory, \"{O}stlund \cite{ostlund} raised a problem that deformations of type $\ri$ and type $\riii$ would be sufficient to describe a homotopy from every generic immersion $S^1 \to \mathbb{R}^2$ to the simple closed curve.  In the same paper \cite{ostlund}, from viewpoints of mathematical physics and topology, \"{O}stlund proved that if any Vassiliev-type function is invariant under $\Omega_1$ and $\Omega_3$, then it is invariant under $\Omega_2$.
It was an evidence to support the above problem.

%

In 2014, Hagge and Yazinski \cite{HY} found a counterexample to this problem with $16$ double points.
Recently, in \cite{IT}, Takimura and the second author of this paper obtained a counterexample with $15$ double points and infinitely many counterexamples, and they further showed that there exist infinitely many equivalence classes of spherical curves under the relation consisting of deformations of type $\ri$ and type $\riii$ up to ambient isotopy.
For more details, see \cite{ITTT}.
However, their arguments work for a restricted class of spherical curves.
It is still difficult to detect whether a given pair of spherical curves are equivalent under the relation consisting of deformations of type $\ri$ and type $\riii$ up to ambient isotopy.
 In \cite{FHIKM}, a deformation of spherical curve, called deformation of type $\alpha$ (Figure~\ref{alpha}) which is a combination of deformations of type $\ri$ and type $\riii$, was introduced and it was shown that if two spherical curves $P$ and $P'$ are equivalent under the relation consisting of deformations of type $\ri$ and type $\riii$ up to ambient isotopy, and satisfy certain technical conditions, then $P'$ is obtained from $P$ by a finite sequence of deformations of type $\alpha$.
In Section~\ref{SecPreliminaries} of this paper, we introduce another type of deformations of sperical curves, called deformation of type $\beta$ (Figure~\ref{beta}).
The main result of this paper (Theorem~\ref{prop1}) shows that the three deformations (deformations of type $\riii$, type $\alpha$ and type $\beta$) are used to describe the equivalence class.
For the statement of Theorem~\ref{prop1}, we introduce one terminology.

A spherical curve $Q$ is called \emph{\ri-minimal} if $Q$ does not contain a $1$-gon.
For a deformation of type $\ri$ in Figure~\ref{reidemei}, we say that the type of the deformation from the left (the right, resp.) to the right (the left, resp.) is of type $\ri^{+}$ ($\ri^{-}$, resp.).
For a spherical curve $Q$, let $\reduced(Q)$ be an $\ri$-minimal spherical curve obtained from $Q$ by successively applying deformations of type $\ri^{-}$.

\begin{theorem}\label{prop1}   
Two spherical curves $P$ and $P'$ are equivalent under (possibly empty) deformations of type $\ri$ and a single deformation of type $\riii$ up to ambient isotopy if and only if $\reduced(P)$ and $\reduced(P')$ are transformed each other by exactly one deformation which is of type $\riii$, type $\alpha$, or type $\beta$ up to ambient isotopy.
\end{theorem}

\begin{corollary}\label{thm1}
Let $P$ and $P'$ be two spherical curves.  Then, $P$ and $P'$ are equivalent under deformations of type $\ri$ and type $\riii$ up to ambient isotopy if and only if 
there is a sequence of $\ri$-minimal spherical curves $P_{0}(=\reduced(P)), P_{1},\dots,P_{n}(=\reduced(P'))$
such that $P_{i+1}$ is obtained from $P_{i}(i=0,1,\dots,n-1)$ by a deformation of type $\riii$, type $\alpha$, or type $\beta$ up to ambient isotopy.

\end{corollary}


We note that Kobayashi-Kobayashi \cite{KK} introduced a new invariant called a stable double point number of pairs of spherical curves and proved that the invariant is non-trivial in a sense by using Theorem~\ref{prop1}.


\section{Preliminaries}\label{SecPreliminaries}
In this section, we introduce some terminologies, notations and facts related to this work.
We start with the following fact.

\begin{fact}[\cite{ITw, K}]\label{factITw}
For any spherical curve $P$, any pair of $\ri$-minimal spherical curves obtained from $P$ by applying deformations of type $\ri^{-}$ are mutually ambient isotopic.
\end{fact}

A spherical curve is called {\it trivial} if it is a simple closed curve on $S^{2}$.
Let $P$ be a non-trivial spherical curve and $m_1$, $m_2$, \ldots, $m_k$ the double points of $P$.
Each component of $P \setminus \left({\displaystyle \bigcup^{k}_{i=1} m_i}\right)$ is called an \emph{arc}.

\begin{definition}[spherical curve with dots]
Let $Q$ be a non-trivial spherical curve with (possibly empty) specified points, ${\{p_{i}\}}_{i=1}^{r}$ called {\it dot(s)} such that: each $p_{i}$ is not a double point; and each arc contains at most one dot.
Then, $Q$ is called a \emph{spherical curve with dots ${\{p_{i}\}}_{i=1}^{r}$}.
\end{definition}

We note that a spherical curve sometimes denotes that a spherical curve with dots.

\begin{definition}[$\ri$-minimal spherical curve with dots]
Let $Q$ be a spherical curve with dots ${\{p_{i}\}}_{i=1}^{r}$.
We say that {\it the spherical curve with dots ${\{p_{i}\}}_{i=1}^{r}$ is $\ri$-minimal} if the boundary of each $1$-gon of $Q$ contains (exactly one) dot.

\end{definition}

It is easy to see that for each spherical curve $Q$ with dots ${\{p_{i}\}}_{i=1}^{r}$, we can obtain an $\ri$-minimal spherical curve with dots by applying deformations of type $\ri^{-}$ for $1$-gons without dots. 
Further it is easy to see that (cf. Fact~\ref{factITw}) such spherical curve with dots are mutually ambient isotopic as spherical curve with dots.


\begin{definition}[connected sum]
\label{connected}
Let $P$ ($P'$, resp.) be a spherical curve with dot(s) $\{p_{i}\}$ ($\{p'_{j}\}$, resp.), where $\{p_{i}\}$ means ${\{p_{i}\}}_{i=1}^{r}$ ($\{p_{j}\}$ means ${\{p_{j}\}}_{j=1}^{r'}$, resp.).
Let $d$ ($d'$, resp.) be a sufficiently small disk with the center $p_1$ ($p'_{1}$, resp.) where $d \cap P$ ($d' \cap P'$, resp.) consists of an arc properly embedded in $d$ ($d'$, resp.).
Let $\widehat{d}=cl(S^2 \setminus d)$ ($\widehat{d'}=cl(S^2 \setminus d'$), resp.) and $\widehat{P}=P \cap \widehat{d}$ ($\widehat{P'}=P' \cap \widehat{d'}$, resp.); let $h : \partial \widehat{d}\to\partial \widehat{d'}$ be a  homeomorphism such that $h(\partial \widehat{P})=\partial \widehat{P'}$ where $h$ is an orientation reversing or preserving homeomorphism
 (i.e., we consider the possibilities orientations with respect to both $S^{1}$ and $S^{2}$).
Then, we have a spherical curve $\widehat{P} \cup_h \widehat{P'}$ with dots $(\{p_{i}\}\setminus\{p_{1}\})\cup(\{p'_{j}\}\setminus\{p'_{1}\})$ (if $r>1$ or $r'>1$), or a spherical curve $\widehat{P} \cup_h \widehat{P'}$ without a dot (that is a canonical shperical curve) (if $r=1$ and $r'=1$) in the oriented $2$-sphere $\widehat{d} \cup_h \widehat{d'}$.
The spherical curve $\widehat{P} \cup_h \widehat{P'}$ with dots $(\{p_{i}\}\setminus\{p_{1}\})\cup(\{p'_{j}\}\setminus\{p'_{1}\})$ or without a dot in the oriented $2$-sphere is denoted by $P \sharp_{(p_1,~p'_1), h} P'$ 
(or $P \sharp_{(p_1,~p'_1)} P'$)
and is called a {\it{connected sum}} of the spherical curves $P$ (with dot(s) $\{p_{i}\}$) and $P'$ (with dot(s) $\{p'_{j}\}$) at the pair of dots $p_1$ and $p'_1$ (see Figure~\ref{connect}).
The circle corresponding to $\partial\widehat{P}(=\partial\widehat{P'})$ is called the {\it decomposing circle} of the connected sum.


\begin{figure}[htbp]
\includegraphics[width=8cm]{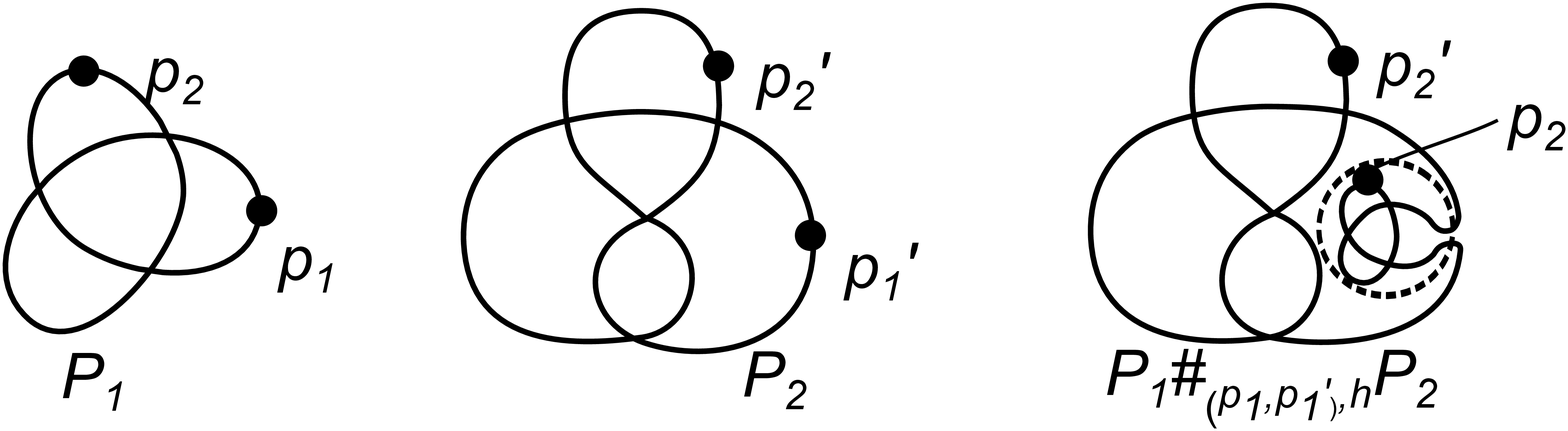}
\caption{An example of a connected sum $P \sharp_{(p_1,~p'_1), h} P'$ of two spherical curves $P$ with dots $\{p_{1},p_{2}\}$ and $P'$ with dots $\{p'_{1},p'_{2}\}$.
In the right figure, the dotted circle denotes the decomposing circle which means $\partial\widehat{P_{1}}(=\partial\widehat{P_{2}})$.}
\label{connect}
\end{figure}
\end{definition}

A spherical curve $P$ is {\it prime} if $P$ is nontrivial and is not a connected sum of two nontrivial spherical curves.
It is elementary to show that any spherical curve admits a prime connected sum decomposition, and hence a set of mutually disjoint decomposing circles corresponding to a prime connected sum decomposition.

\begin{definition}[deformations of type $\alpha$, deformation of type $\beta$]\label{alpha_beta}
For spherical curves $P$ and $P'$, we say that $P'$ is obtained from $P$ by a {\it deformation of type $\alpha$}, if $P'$ is obtained by replacing the part of $P$ contained in a disk as in Figure~\ref{alpha}.
We say that $P'$ is obtained from $P$ by a deformation of type $\alpha^{+}$ (type $\alpha^{-}$, resp.), if the number of double points of $P'$ is greater (less, resp.) than that of $P$.
See Figure~\ref{alpha}.

\begin{figure}[h!]
\centering
\includegraphics[width=6cm]{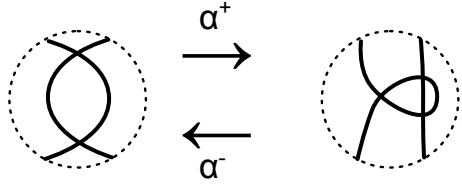}
\caption{deformation of type $\alpha$
}\label{alpha}
\end{figure}

Let $3_1$ be a prime spherical curve with exactly three double points with a dot $\{p\}$.
Then, we say that $P'$ is obtained from $P$ by a deformation of type $\beta^{+}$ or $\beta^{+}(m)$, where $m$ is a non-negative integer, if $P'$ is a connected sum $P \sharp_{(q,~p_{2})} ({\infty}^m \sharp_{(p_{1},p),h} 3_1)$, where $\{q\}$ is a dot in $P$, ${\infty}^m$ denotes an $\ri$-minimal spherical curve with two dots $\{p_{1}, p_{2}\}$ obtained from a trivial spherical curve by applying $m$ deformations of type $\ri^{+}$.  
Then a connected sum ${\infty}^m \sharp_{(p_{1},p), h} 3_1$ is $\ri$-minimal spherical curve with a dot $\{p_{2}\}$. 
We say that $P$ is obtained from the connected sum $P'=P \sharp ({\infty}^m \sharp 3_1)$ by the deformation of type $\beta^{-}$ or $\beta^{-}(m)$.
\end{definition}

By Definition~\ref{alpha_beta}, we may regard a deformation of type $\beta^{\pm}(m)$ as a deformation fixing an ambient 2-sphere.

\noindent
{\it Remark.} 
We note that $\alpha^{+}$ is expressed as a combination of a deformation of type $\ri^{+}$ and a deformation of type $\riii$, and that $\beta^{+}(m)$ is expressed as a combination of $m+3$ deformations of type $\ri^{+}$ and a deformation of type $\riii$.


\medskip
\noindent
{\bf Example.} Let $4_{1}$ be the spherical curve with a dot $\{p\}$ as in Figure~\ref{beta}~(a).
Then Figure~\ref{beta}~(b) is the list of spherical curves obtained from $4_{1}$ (with a dot $\{p\}$) by applying $\beta^{+}(2)$ up to ambient isotopy.

\begin{figure}[htbp]
\centering
\includegraphics[width=10cm]{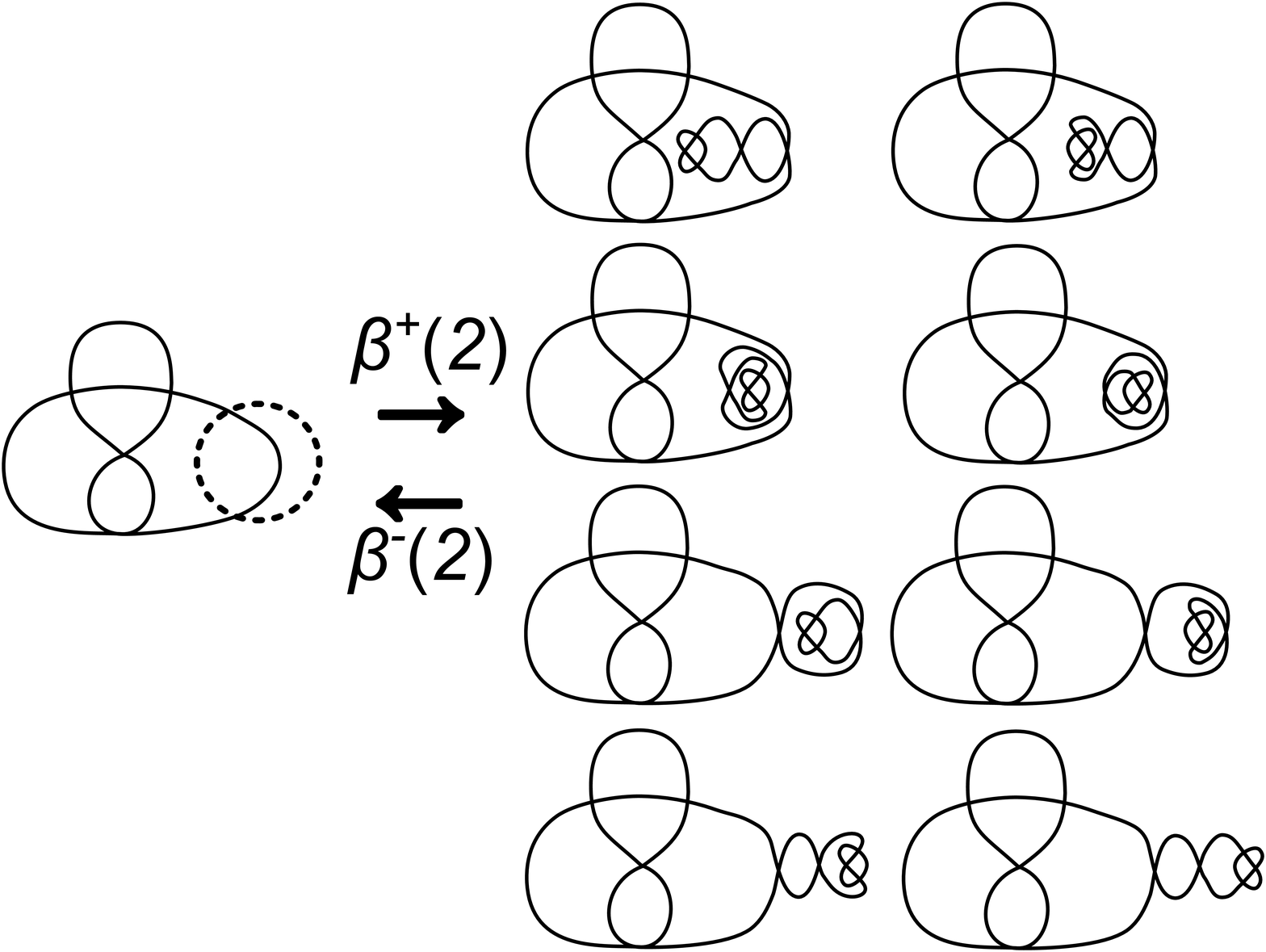}\\
 \ \ \ \ \ \ \ \ \ (a)\ \ \ \ \ \ \ \ \ \ \ \ \ \ \ \ \ \ \ \ \ \ \ \ \ \ \ \ \ \ \ \ \  (b)\ \ \ \ \ \ \ \ \ \ \ \ \ \ \ \ \ \ \ \ \ \ \ \ \ \ \ \ 
\caption{$\beta^{+}(2)$ and $\beta^{-}(2)$}\label{beta}
\end{figure}

\begin{definition}
Let $P$ be a spherical curve.
A double point $d$ is said to be {\it nugatory} if there exists a circle on $S^2$ which transversely intersects $P$ only in $d$ (Figure~\ref{nugatory}).
\end{definition}

\begin{figure}[htbp]
\centering
\includegraphics[width=3cm]{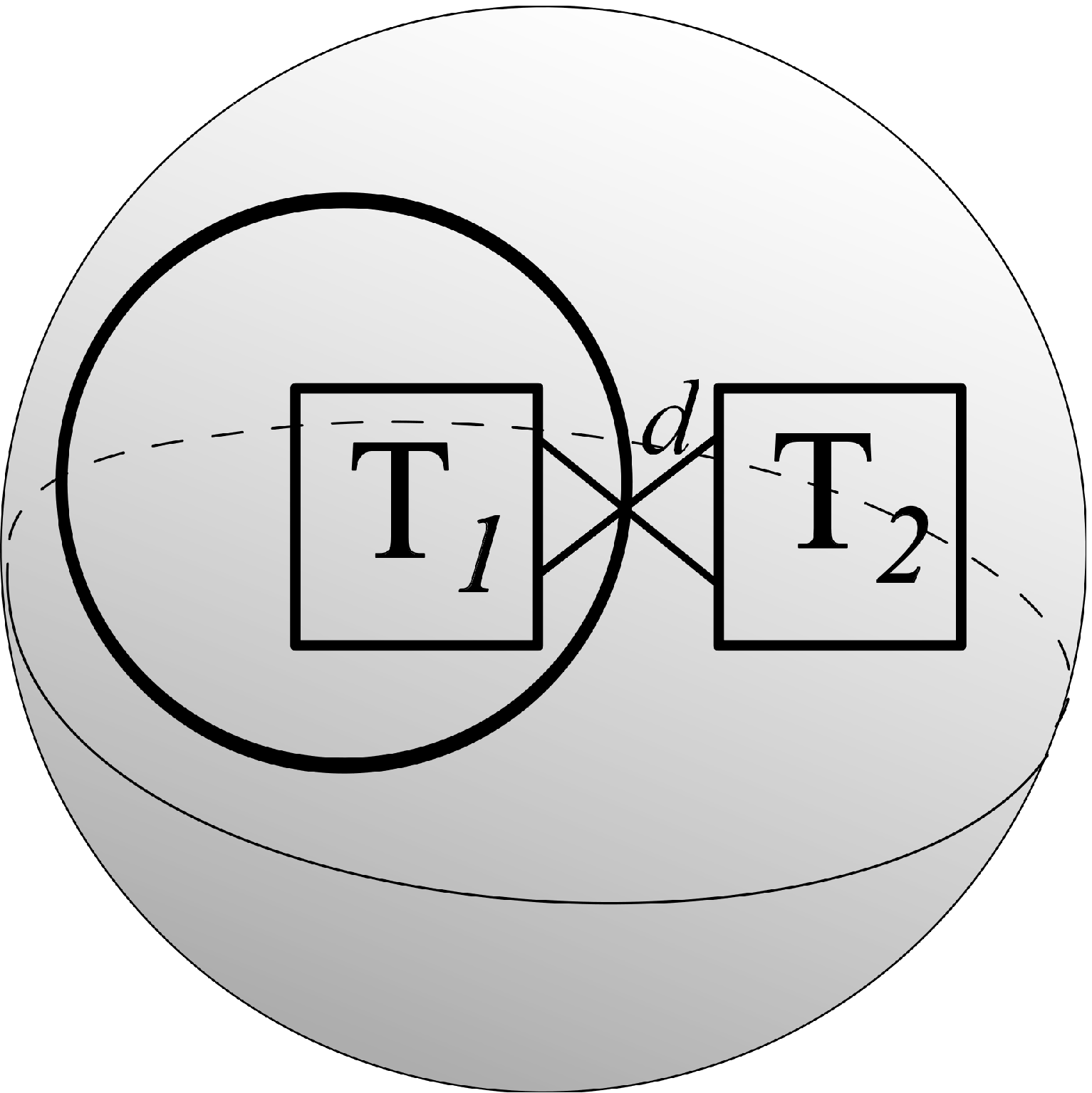}
\caption{}
\label{nugatory}
\end{figure}

\begin{notation}[$f_c$]\label{notation1}
For a spherical curve $P$, let $f_c(P)$ be the number of prime factors of $P$.
\end{notation} 

Let $Q$ and $Q'$ be non-trivial spherical curves.
Suppose that $Q'$ is obtained from $Q$ by a single deformation of type $\alpha^{+}$.
Let $d_1$, and $d_2$ be the double points of $Q$ relevant to the deformation of type $\alpha^{+}$.
Then we have Lemmas~\ref{prime_alpha} and \ref{difference_fc_alpha} below.
Lemmas~\ref{prime_alpha} and \ref{difference_fc_alpha} concerned with prime factors of $Q$.

\begin{lemma}\label{prime_alpha}
Suppose that $Q$ is prime.
Then $Q'$ is prime.
\end{lemma}

\begin{proof}
For a contradiction, suppose that $Q'$ is not prime.
Let ${\mathcal C}$ be a set of circles that divides prime factors of $Q'$.
Let $D$ be a small (open) disk ($\subset S^{2}$) where the deformation of type $\alpha^{+}$ is applied.
We note that $D$ is a subsurface of $S^{2}$.
Hence, since $Q$ is applied the deformation of type $\alpha^{+}$, there are $d_{1}$ and $d_{2}$ in $D$ as in the left figure of Figure~\ref{alpha} first.
After $Q'$ was obtained from $Q$ by the deformation of type $\alpha^{+}$ within $D$, there are three double points of $Q'$ in $D$ as in the right figure of Figure~\ref{alpha}.
Then we claim that $D\cap {\mathcal C}=\emptyset$.
For a contradiction, suppose that there exists $C\in{\mathcal C}$ such that $C\cap D\neq\emptyset$.
Then $C\cap (D\cap Q')$ looks within $D$ as in Figure~\ref{Fact_prime_alpha} since $C\cap (D\cap Q')$ consists of at most two points.
However if $C\cap (D\cap Q')$ is as in Case~{\bf a}, then there is a trivial factor for the prime decomposition of $Q'$, a contradiction.
If $C\cap (D\cap Q')$ is as in Case~{\bf b}, then it is easy to see that $Q'$ is a union of more than one spherical curves, a contradiction.
Hence $D\cap {\mathcal C}=\emptyset$.
Let $d'_{1}, d'_{2}, d'_{3}$ be double points of $Q'$ relevant to the deformation of type $\alpha^{+}$. 
Since $D\cap {\mathcal C}=\emptyset$, we see that $d'_{1}, d'_{2}, d'_{3}$ are contained in a prime factor.
Hence, we may regard ${\mathcal C}$ as a system of decomposing circles for $Q$.
Obviously, ${\mathcal C}$ gives a non-trivial connected sum decomposition of $Q$, a contradiction.

\begin{figure}[htbp]
\centering
\includegraphics[width=6cm]{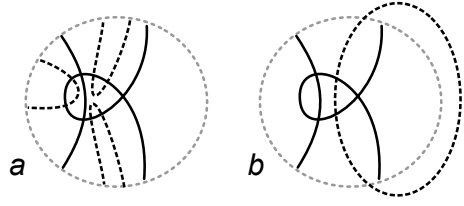}
\caption{The dotted lines and loop mean elements of ${\mathcal C}$.}
\label{Fact_prime_alpha}
\end{figure}

\end{proof}

\begin{lemma}\label{difference_fc_alpha}
Suppose that $d_{1}$ and $d_{2}$ are not nugatory.
Then, $f_c(Q')=f_c(Q)$.
\end{lemma}

\begin{proof}
Let ${\mathcal C}$ be a set of circles that divides prime factors of $Q$.
Let $D$ be a small disk where the deformation of type $\alpha^{+}$ is applied.
First we show that $D\cap {\mathcal C}=\emptyset$.
For a contradiction, suppose that $D\cap {\mathcal C}\neq\emptyset$.
Let $C$ be a component of  $\mathcal C$ such that $D\cap C\neq\emptyset$.
Since $Q\cap C$ consists of two points, $C\cap (D\cap Q)$ consists of at most two points.
%
%
Let $a_{1},a_{2}$ be the edges of the $2$-gon (hence $\partial a_{i}=d_{1}\cup d_{2}$).
Here, we note that $C\cap a_{1}\neq\emptyset$, and $C\cap a_{2}\neq\emptyset$, because if $C\cap a_{i}=\emptyset$, then we see that a prime factor of $P$ decomposed by ${\mathcal C}$ is trivial (see Figure~\ref{trivialFactorCotra}).
Hence $C\cap D$ will look as in Figure~\ref{2gonTransverse}.
By using an isotopy along the triangle in Figure~\ref{2gonTransverse}, we can obtain a circle $C^{\ast}$ from $C$ such that $C^{\ast}\cap P$ consists of  a single transverse point $d_{i}$ ($i=1,2$), hence $d_{1}$ and $d_{2}$ are nugatory, a contradiction.

Finally, we show that $f_c(Q')=f_c(Q)$.
Let $Q_{0}$ be a prime factor including $d_{1}$ and $d_{2}$.
Then $Q$ admits the connected sum decomposition $Q=Q_{0}\sharp Q_{1}\sharp \cdots\sharp Q_{s}$ (possibly $Q_{1}\sharp \cdots\sharp Q_{s}$ is empty) where $Q_{1},\dots, Q_{s}$ are prime factors.
Since $D\cap{\mathcal C}=\emptyset$, $Q'$ admits a connected sum decomposition $Q'=Q'_{0}\sharp Q_{1}\sharp \cdots\sharp Q_{s}$, where $Q'_{0}$ is obtained from $Q_{0}$ by a deformation of type $\alpha^{+}$ performed within $D$.
These facts together with Lemma~\ref{prime_alpha}, shows that $Q'_{0}$ is prime.
Therefore $f_c(Q')=f_c(Q)$ ($=s+1$).
\end{proof}
\begin{figure}[htbp]
\centering
\includegraphics[width=2.8cm]{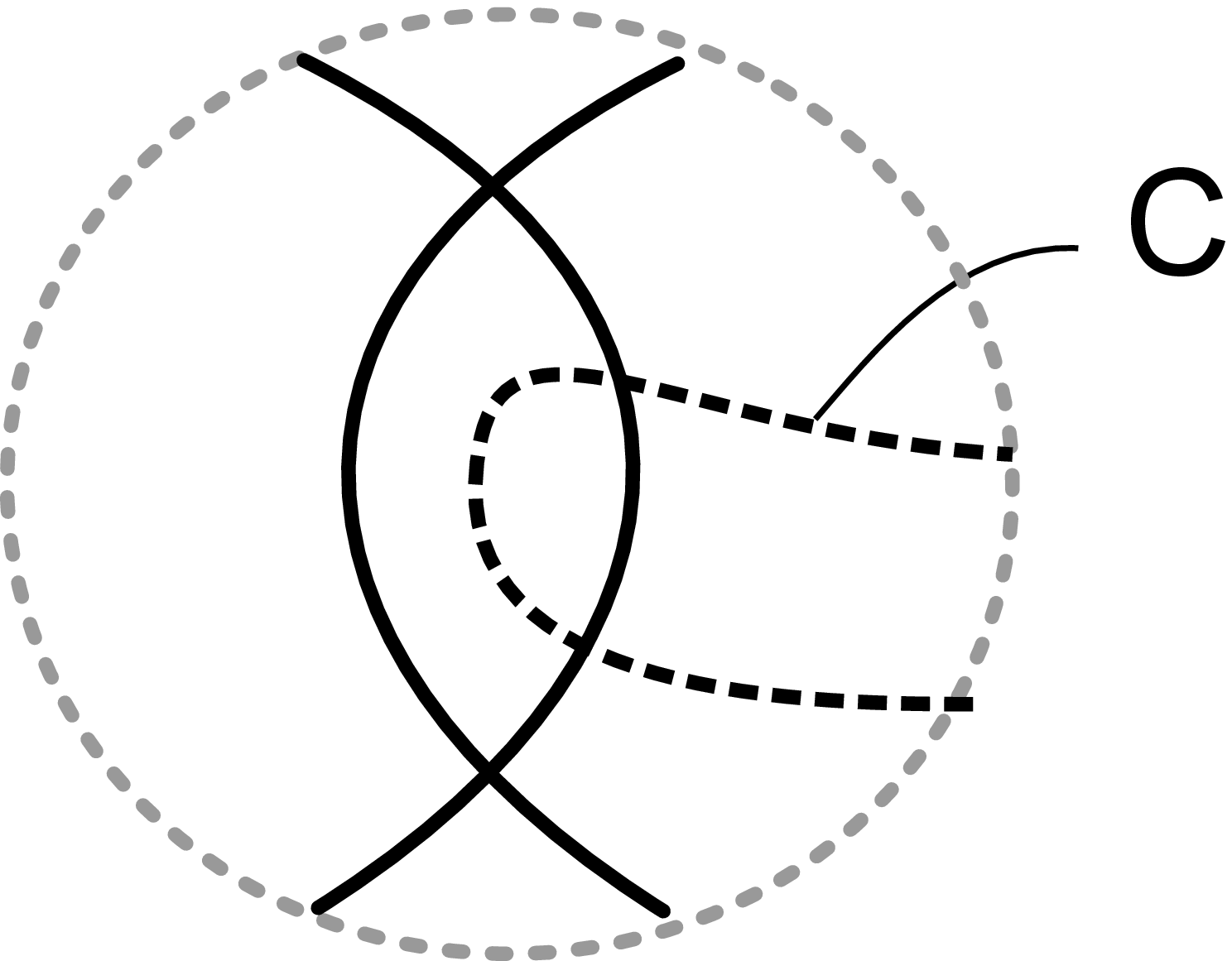}\caption{}
\label{trivialFactorCotra}
\end{figure}

\begin{figure}[htbp]
\centering
\includegraphics[width=7cm]{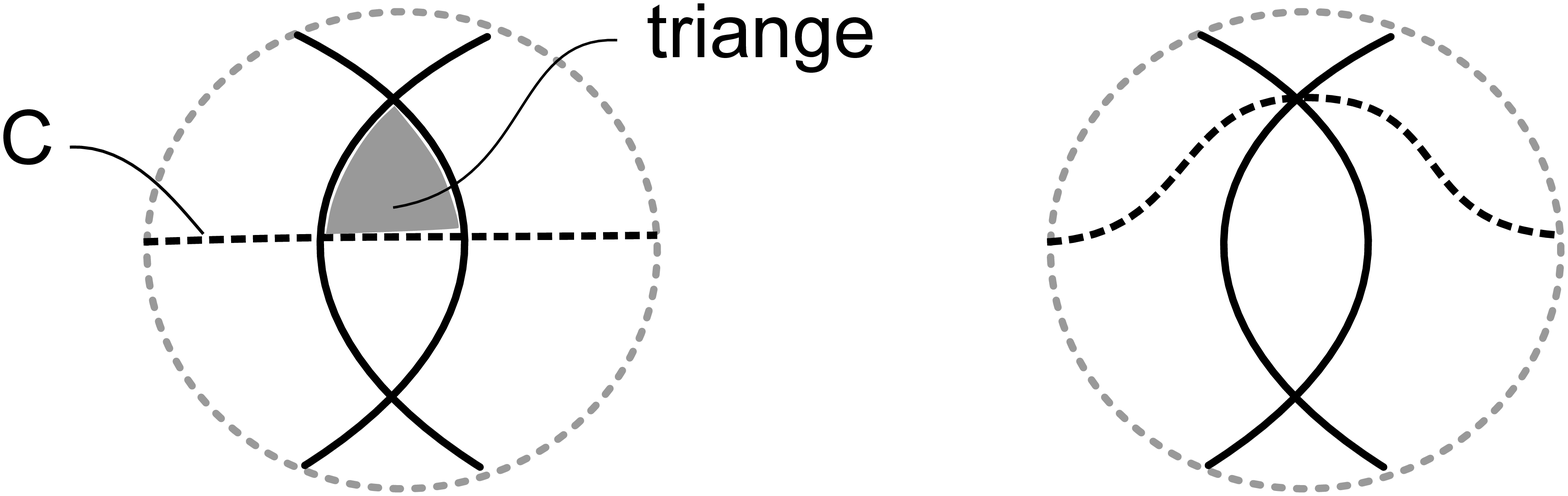}\caption{}
\label{2gonTransverse}
\end{figure}

Let $P$ and $P'$ be spherical curves.
Suppose that $P'$ is obtained from $P$ a single deformation of type $\riii$.
Let $c_1$, $c_2$ and $c_3$ be the three double points of $P$ relevant to the deformation of type $\riii$.
Then we have Lemma~\ref{deference_fc} below.

\begin{lemma}\label{deference_fc}
If some $c_{i}$ is nugatory, then $f_c(P')<f_c(P)$ (in particular, $f_c(P')=f_c(P)+1$ or $f_c(P')=f_c(P)+2$).
\end{lemma}

\begin{proof}
Note that we have essentially the following two cases.

\medskip
\noindent
{\bf Case~1}: 
Exactly one of $c_{1},c_{2},c_{3}$, say $c_{1}$ is nugatory.\\
Since $c_{1}$ is nugatory, $P$ admits a connected sum decomposition $P_{1}\sharp \infty\sharp P_{2}$ where $P_{1}$ and $P_{2}$ are obtained from $P$ by smoothing at $c_{1}$ and $\infty$ means $\infty^{1}$ in Definition~\ref{alpha_beta}.
Hence $f_c(P)=f_c(P_{1})+1+f_c(P_{2})$.

On the other hand, by Figure~\ref{c1isNugatory}, it is directly observed that $P'$ admits a connected sum decomposition $P_{1}\sharp P'_{2}$ where $P'_{2}$ is obtained from $P_{2}$ by a deformation of type $\alpha^{+}$.
Hence $f_c(P')=f_c(P_{1})+f_c(P'_{2})$.
%

Since $c_{2}$ and $c_{3}$ are not nugatory, by Lemma~\ref{difference_fc_alpha}, we see that $f_c(P'_{2})=f_c(P_{2})$.
Hence $f_c(P)=f_c(P_{1})+1+f_c(P_{2})=f_c(P_{1})+1+f_c(P'_{2})=f_c(P')+1$.
Then we see $f_c(P')<f_c(P)$.

\begin{figure}[htbp]
\centering
\includegraphics[width=8cm]{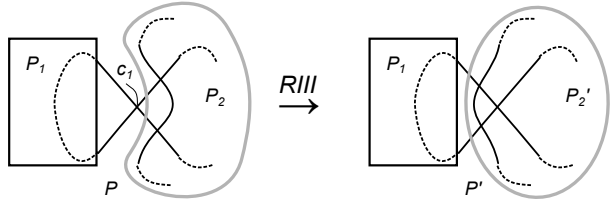}\caption{$P$ admits a connected sum decomposition $P_{1}\sharp \infty\sharp P_{2}$, whereas $P'$ admits a connected sum decomposition $P_{1}\sharp P'_{2}$.}
\label{c1isNugatory}
\end{figure}

\medskip
\noindent
{\bf Case~2}: 
At least two of $c_{1},c_{2},c_{3}$, say $c_{1},c_{2}$, are nugatory.\\
Since $c_{1},c_{2}$ are nugatory, it is easy to see that $c_{3}$ is nugatory also.
(Recall the relationship between spherical curves with dots and connected sums as in Definition~\ref{connected}.)
Let ${\widetilde \infty^{3}}$ be a spherical curve with dots as Figure~\ref{tildeInfty3}.
Then $P$ admits a connected sum decomposition $P_{1}\sharp P_{2}\sharp P_{3}\sharp {\widetilde \infty^{3}}$.
Then $f_c(P)=f_c(P_{1})+f_c(P_{2})+f_c(P_{3})+3$ since $f_c({\widetilde \infty^{3}})=3$.

On the other hand, since $3_{1}$ is obtained from ${\widetilde \infty^{3}}$ by a deformation of type $\riii$, $P'$ admits a connected sum decomposition $P_{1}\sharp P_{2}\sharp P_{3}\sharp 3_{1}$, where $3_{1}$ is as in Definition~\ref{alpha_beta}.
Then $f_c(P')=f_c(P_{1})+f_c(P_{2})+f_c(P_{3})+f_c(3_{1})=f_c(P_{1})+f_c(P_{2})+f_c(P_{3})+1$.
Hence $f_c(P)=f_c(P')+2$.
Then we see $f_c(P')<f_c(P)$.

\begin{figure}[htbp]
\centering
\includegraphics[width=3.5cm]{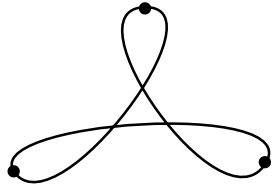}\caption{${\widetilde \infty^{3}}$}
\label{tildeInfty3}
\end{figure}

\end{proof}

\begin{notation}\label{notationOpSeq}
Let $P$ and $P'$ be two spherical curves that are related by a finite sequence of deformations of type $\ri$ and type $\riii$, i.e., there exists a finite sequence of spherical curves $P = P_0, P_1,  \dots, P_r = P'$, where $P_{i}$ is obtained from $P_{i-1}$ by a single deformation of type $\ri$ or type $\riii$.
Then,  $Op_i$ denotes the deformation from $P_{i-1}$ to $P_{i}$, and these settings are expressed by   
\[P = P_0 \stackrel{Op_1}{\to} P_1 \stackrel{Op_2}{\to} \cdots \stackrel{Op_r}{\to} P_r = P'.\]  
\end{notation}

In the reminder of this section, for the proof of Theorem~\ref{prop1}, we prepare Claim~\ref{cycle}. 
Let $Q$ and $Q'$ be spherical curves.
Suppose that $Q'$ is obtained from $Q$ a single deformation of type $\riii$ and $f_c (Q') \ge f_c (Q)$.
Let $D$ be a small disk where the deformation of type $\riii$ is applied.
Let $c_1$, $c_2$ and $c_3$ be the double points in $Q\cap D$.
Since $f_c (Q') \ge f_c (Q)$, every $c_{\lambda}$ is not a nugatory double point ($\lambda=1,2,3$) by Lemma~\ref{deference_fc}.

On the other hand, by Fact~\ref{factITw}, there exists a finite sequence of deformations of type $\ri^{-}$ from $Q$ to $\reduced(Q)$.
Since $c_1$, $c_2$ and $c_3$ are not nugatory, this sequence does not affect to $Q\cap D$.
Then, let $R$ be the spherical curve obtained from $\reduced(Q)$ by the single deformation of type $\riii$ within $D$.

Moreover, since the sequence consisting of deformations of type $\ri^{-}$ does not affect to $Q\cap D$, applying the sequence of deformations of type $\ri^{-}$ to $Q'$ makes sense.
Let $\widetilde{Q'}$ be the spherical curve obtained from $Q'$ by applying the sequence.
Then we have following claim.

\begin{claim}\label{cycle}
$R$ and $\widetilde{Q'}$ are ambient isotopic.
\end{claim}

\begin{proof} 
Since $R$ is obtained from $\reduced(Q)$ by applying a deformation of type $\riii$, we see that $R\cap \overline{D}=\reduced(Q)\cap \overline{D}$, where $\overline{D}=cl(S^2 \setminus D)$.
Since $\reduced(Q)\cap \overline{D}=\widetilde{Q'}\cap\overline{D}$, we have that $R\cap \overline{D}=\widetilde{Q'}\cap \overline{D}$.

Moreover, $Q'$ is obtained from $Q$ by applying the deformation of type $\riii$ to $Q\cap D$ and $R$ is obtained from $\reduced(Q)$ by applying the deformation of type $\riii$ to $\reduced(Q)\cap D$.
Since the sequence of deformations of type $\ri^{-}$ dose not affect $D$, these facts imply $R\cap D=\widetilde{Q'}\cap D$.

Thus, $R$ and $\widetilde{Q'}$ are ambient isotopic.

\end{proof}

%
%
%


\section{Proof of Theorem~\ref{thm1}}\label{sec3}

In this section, we give a proof of Theorem~\ref{prop1}. 


\noindent{\bf{Proof of if part of the statement of Theorem~\ref{thm1}.}}

By Remark of Definition~\ref{alpha_beta}, a deformation of type $\alpha$ ($\beta^{\pm}(m)$, resp.) consists of a single deformation of type $\riii$ and a single deformation of type $\ri$ ($3+m$ deformations of type $\ri$, resp.).
This together with the definition of $\reduced(\cdot)$ implies if part of the statement of Theorem~\ref{prop1}.

\noindent{\bf{Proof of only if part of the statement of Theorem~\ref{thm1}.}}


By the assumption, there exists a finite sequence denoted by;
\[P = P_0 \stackrel{Op_1}{\to} P_1 \stackrel{Op_2}{\to} \cdots \stackrel{Op_r}{\to} P_r = P'\]
where $\{ Op_i \}_{i=1}^{r}$ consists deformations of type $\ri$ and a single deformation of type $\riii$ (Notation~\ref{notationOpSeq}).
Let $j$ be the integer such that $Op_j$ is the deformation of type $\riii$.

By exchanging $P$ and $P'$, if necessary, we may suppose that $f_{c}(P_{j})\ge f_{c}(P_{j-1})$.
Hence it is clear that the next claim gives the proof of the only if part of the statement of Theorem~\ref{thm1}.
Hence in the remainder of this section, we give a proof of Claim~\ref{claimA}.

\begin{claim}\label{claimA}
If $f_{c}(P_{j})\ge f_{c}(P_{j-1})$, then $\reduced(P)$ is obtained from $\reduced(P')$ by a single deformation of type $\riii$, type $\alpha^-$ or type $\beta^-$.  
\end{claim}

Let $D$ be a small disk to which $Op_j$ is applied.
Let $f$ be the generic immersion satisfying that $f(S^1)=P_{j-1}$.
%
%
By definition, since $D$ is a disk containing a $3$-gon corresponding to a deformation of type $\riii$, $f^{-1}(D)$ consists of three components $J_{1}$, $J_{2}$ and $J_{3}$.
Then the number of components of $S^1 \setminus f^{-1}(D)$ is also three.
The three components of $S^1 \setminus f^{-1}(D)$ are denoted by $I_1$, $I_2$, and $I_3$.
%
%
%
%
Let $c_1$ ($c_2$, $c_3$, resp.) be the double point that is the intersection of $f(J_2)$ and $f(J_3)$ ($f(J_3)$ and $f(J_1)$, $f(J_1)$ and $f(J_2)$, resp.).
Let $\delta$ be the number of the components of $P_{j-1}\setminus D=P_{j-1} \cap \overline{D}$, where $\overline{D}=cl(S^2 \setminus D)$.

Let $g$ be a generic immersion satisfying that $g(S^1)=\reduced(P_{j-1})$ such that $S^1 \setminus g^{-1}(D)$ consists of the tree components $I_{1},I_{2}$ and $I_{3}$.

\medskip
\noindent
{\bf Case~1}: $\delta=1$ (i.e., $f(I_{\lambda}) \cap f(I_{\mu})$ $\neq$ $\emptyset$ and $f(I_{\mu}) \cap f(I_{\nu})$ $\neq$ $\emptyset$, where $\lambda$, $\mu$ and $\nu$ are mutually distinct numbers in $\{1, 2, 3 \}$).

Since $\delta=1$, we see that each $c_{i}$ is not nugatory.
This together with Fact~\ref{factITw} shows that the sequence of the deformations of type $\ri^{-}$ from $P_{j-1}$ to $\reduced(P_{j-1})$ dose not affect $D$.
Hence we may apply $Op_{j}$ to $\reduced(P_{j-1})$, and $S_{\iii}$ denotes the obtained spherical curve.

%
%
%

\begin{claim}\label{claim4}
$S_{\iii}$ is $\ri$-minimal.
\end{claim}
\begin{proof}
By the definition of a deformation of type $\riii$, there is no $1$-gon completely included in $S_{\iii}\cap D$.
Since the sequence from $P_{j-1}$ to $\reduced(P_{j-1})$ resolves $1$-gons in $P_{j-1}\cap\overline{D}$, there is no $1$-gon completely included in $S_{\iii}\cap \overline{D}$. 
If there exists a $1$-gon on the 2-sphere, then the $1$-gon intersects $S_{\iii}\cap D$.
Here, for example, the 1-gon will look as in Figure~\ref{both1gon}.
However, this contradicts the condition of Case~1.

\begin{figure}[htbp]
\centering
\includegraphics[width=2.8cm]{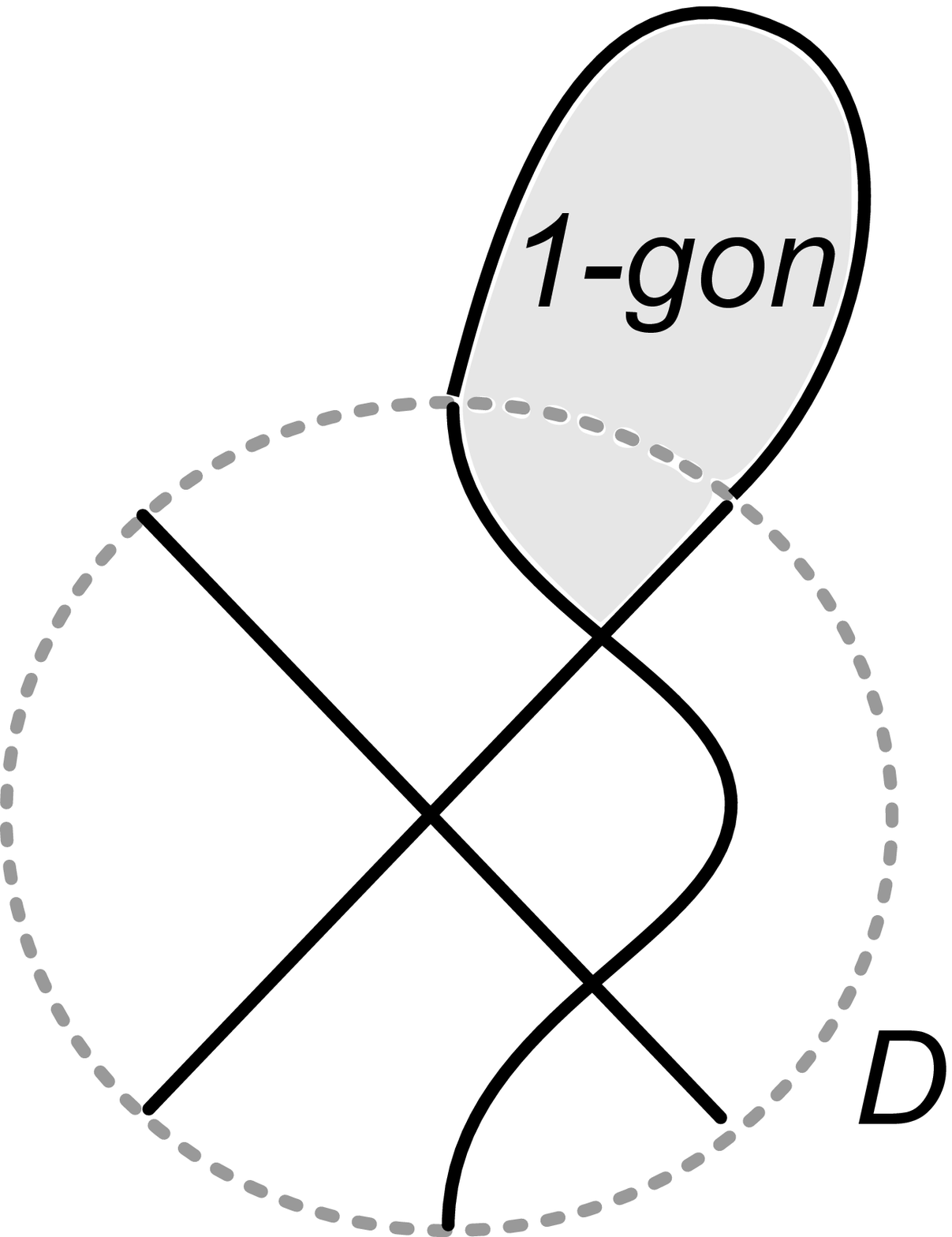}\caption{}
\label{both1gon}
\end{figure}

\end{proof}

Since the sequence of deformations of type $\ri$ from $P_{j-1}$ to $\reduced(P_{j-1})$ dose not affect $D$, we may apply the sequence to $P_{j}$.
Then we denote the obtained spherical curve ${\widetilde P_{j}}$.
Hence, we see that ${\widetilde P_{j}}\cap\overline{D}$ is ambient isotopic to $\reduced(P_{j-1})\cap\overline{D}$ and ${\widetilde P_{j}}\cap D$ is ambient isotopic to $P_{j}\cap D$.
Then we can apply similar arguments of the proof of Claim~\ref{cycle} with regarding $Q=P_{j-1}$, $Q'=P_{j}$, $R=S_{\iii}$ and $\widetilde{P_{j}}=\widetilde{Q'}$, 
it is clear that $S_{\iii}$ is ambient isotopic to ${\widetilde P_{j}}$.
This fact together with Claim~\ref{claim4}, we see that ${\widetilde P_{j}}$ is ambient isotopic to $\reduced(P_{j})$.
On the other hand, the sequence $\{ Op_i \}_{i=j+1}^{r}$ from $P_{j}$ to $P_{r}$ ($=P'$) consist of deformations of type $\ri$, by Fact~\ref{factITw},
it is clear that $\reduced(P')=\reduced(P_{j})$ up to ambient isotopy.
Hence $\reduced(P')={\widetilde P_{j}}$ up to ambient isotopy.

On the other hand, since the sequence $\{ Op_i \}_{i=1}^{j-1}$ from $P$ ($=P_{0}$) to $P_{j-1}$ consists of deformations of type $\ri$, by Fact~\ref{factITw}, it is clear that $\reduced(P)=\reduced(P_{j-1})$ up to ambient isotopy.

%
%
Recall that $S_{\iii}$ is obtained from $\reduced(P_{j-1})$ by applying $Op_j$.
This fact together with the above shows that $\reduced(P')$ is obtained from $\reduced(P)$ by a single deformation of type $\riii$.

\medskip
\noindent
{\bf Case~2}: $\delta=2$ (i.e., $f(I_{\lambda}) \cap f(I_{\mu})$ $\neq$ $\emptyset$, $f(I_{\mu}) \cap f(I_{\nu})$ $=$ $\emptyset$ and $f(I_{\nu}) \cap f(I_{\lambda})$ $=$ $\emptyset$, where $\lambda$, $\mu$ and $\nu$ are mutually distinct numbers in $\{1, 2, 3 \}$).
Without loss of generality, we may suppose that $f(I_1) \cap f(I_2)$ $\neq$ $\emptyset$, $f(I_2) \cap f(I_3)$ $=$ $\emptyset$, and $f(I_3) \cap f(I_1)$ $=$ $\emptyset$.  
 
\medskip
\noindent $\bullet$ Case~2-1: $g(I_3)$ is a simple arc.

By Fact~\ref{factITw}, there exists a sequence of deformations of type $\ri^{-}$ from $P_{j-1}$ to $\reduced(P_{j-1})$.
By the assumption: $f_{c}(P_{j}) \geq f_{c}(P_{j-1})$, we see that $c_{1},c_{2},c_{3}$ are not nugatory (Lemma~\ref{deference_fc}).
Hence $P_{j-1}\cap D$ is not affected by the sequence.
Then, we apply the sequence of deformations of type $\ri^{-}$ to $P_j$ ($\subset S^2$).
Let $\widetilde{P_{j}}$ be the spherical curve obtained from $P_j$ by applying the sequence.
Moreover let $S_{\iii}$ be a spherical curve obtained from $\reduced(P_{j-1})$ by applying a deformation of type $\riii$ to $\reduced(P_{j-1})\cap D$.
Since $f_{c}(P_{j}) \geq f_{c}(P_{j-1})$, we can apply Claim~\ref{cycle} with regarding $Q=P_{j-1}, Q'=P_{j}, R=S_{\iii}$ and $\widetilde{P_{j}}=\widetilde{Q'}$, we have; 

\begin{claim}\label{claim5-2-1}
$\widetilde{P_{j}}$ and $S_{\iii}$ are ambient isotopic.
\end{claim}

By the assumption of Case~2-1, $g(I_3)$ is a simple arc.
Let $E$ be a sufficiently small disk completely including $g(I_3)$ and the $3$-gon of $\reduced(P_{j-1})\cap D$.
We note that the configuration ($E$, $P_{j-1}\cap E$) is the same as that of the right figure of Figure~\ref{alpha} since a simple arc connects a $3$-gon in $E$.
Hence we can apply $\alpha^-$ to $\reduced(P_{j-1})$ in $E$.
Then, we denote 
the resulting spherical curve by $S_{\alpha}$.
Then, by the definition of $\alpha^{-}$, we have;

\begin{claim}\label{S_alphaObtaindedFromS_3}
$S_{\alpha}$ is obtainded from $S_{\iii}$ by a deformation of type $\ri^{-}$ performed within $E$.

\end{claim}

Furthermore, we see Claim~\ref{claim4-2-1} below.

\begin{claim}\label{claim4-2-1}
$S_{\alpha}$ is $\ri$-minimal.
\end{claim}

\begin{proof}
Assume, for a contradiction, that $S_{\alpha}$ admits a 1-gon, denoted by $G$.
By the construction of $S_{\alpha}$, we see that $G$ is not completely included in $\overline{D}$.
Hence $G\cap D\neq\emptyset$, and this shows that $G$ is as the right side in Figure~\ref{1-gonG}.
However, it is easy to see that the configuration implies $g(I_{1})\cap g(I_{2})=\emptyset$, a contradiction.
\end{proof}

\begin{figure}[htbp]
\centering
\includegraphics[width=8cm]{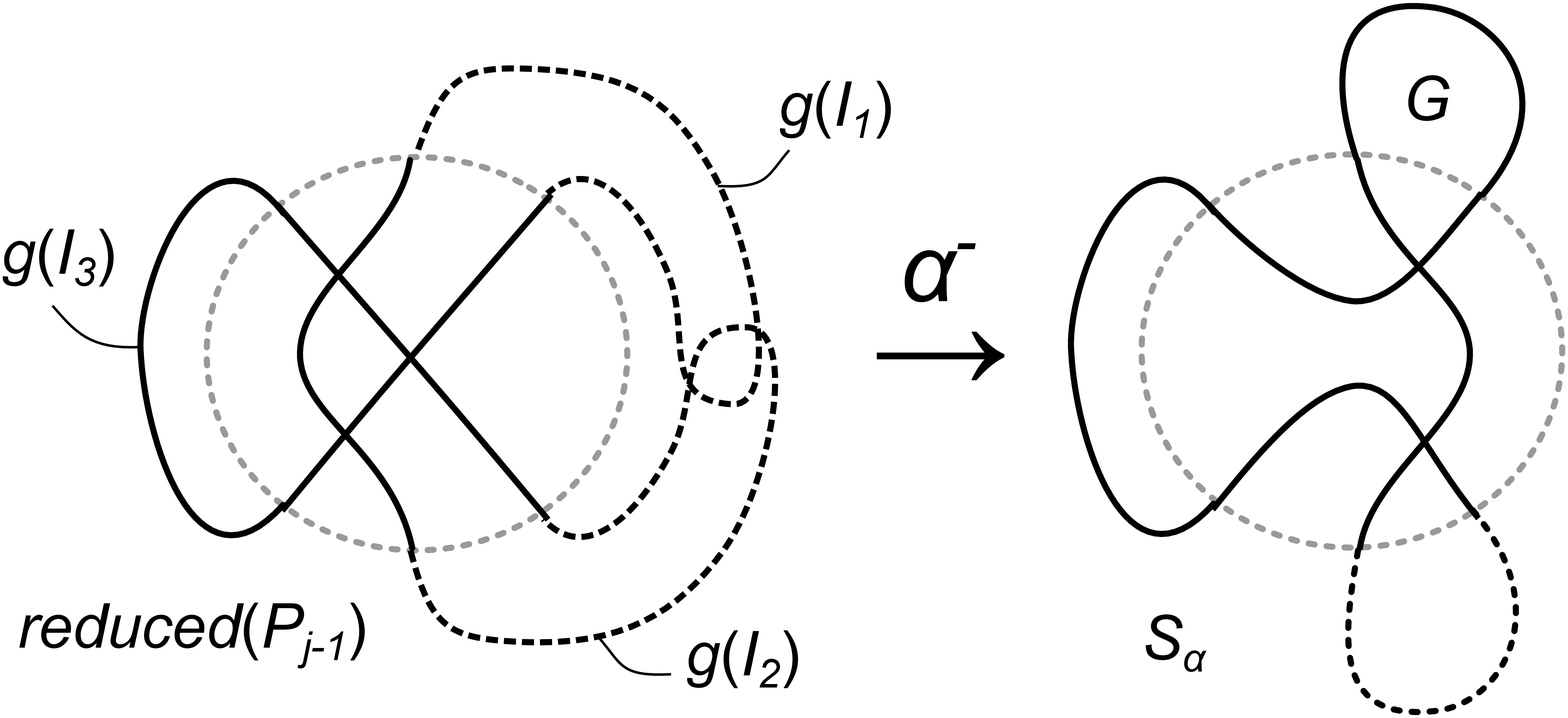}
\caption{}
\label{1-gonG}
\end{figure}


By Fact~\ref{factITw}, $\reduced(P)=\reduced(P_{j-1})$ up to ambient isotopy.
Then $S_{\alpha}$ is regarded as a spherical curve obtained from $\reduced (P)$ by applying a deformation of type ${\alpha^{-}}$ in $E$.
Furthermore, by Claim~\ref{claim4-2-1}, $S_{\alpha}=\reduced (S_{\alpha})$.
Moreover, by Claim~\ref{S_alphaObtaindedFromS_3}, $S_{\alpha}=\reduced(S_{\iii})$ up to ambient isotopy.
By Claim~\ref{claim5-2-1}, $\reduced(S_{\iii})=\reduced(\widetilde{P_j})$ up to ambient isotopy.  
By using Fact~\ref{factITw} and the definition of the sequence from $P_{j-1}$ to $\reduced(P_{j-1})$ again, $\reduced(\widetilde{P_j})=\reduced(P')$ up to ambient isotopy.

These facts show that $\reduced(P')$ is obtained from $\reduced(P)$ by a deformation of type $\alpha^{-}$.

\medskip
\noindent $\bullet$ Case~2-2: $g(I_3)$ is not a simple arc.

After recalling that $g(S^1)=\reduced(P_{j-1})$, in this case, it is easy to show that the arguments in Case~1 work to show that $\reduced(P')$ is obtained from $\reduced(P)$ by a deformation of type $\riii$.
Details of the arguments are left to the reader.


\medskip
\noindent
{\bf Case~3}: $\delta=3$ (i.e., $f(I_{\lambda}) \cap f(I_{\mu})$ $=$ $\emptyset$, $f(I_{\mu}) \cap f(I_{\nu})$ $=$ $\emptyset$ and $f(I_{\nu}) \cap f(I_{\lambda})$ $=$ $\emptyset$, where $\lambda$, $\mu$ and $\nu$ are mutually distinct numbers in $\{1, 2, 3 \}$).

\medskip
\noindent $\bullet$ Case~3-1: For some $\lambda,\mu$ ($\lambda,\mu\in\{1,2,3\},\ \lambda\neq\mu$), each of $g(I_{\lambda})$ and $g(I_{\mu})$ is a simple arc.
Without loss of generality, we may suppose that $(\lambda,\mu)=(1,2)$.

By Fact~\ref{factITw}, there exists a finite sequence of deformations of type $\ri^{-}$ from $P_{j-1}$ to $\reduced(P_{j-1})$.
By the assumption: $f_{c}(P_{j}) \geq f_{c}(P_{j-1})$, we see that $c_{1},c_{2},c_{3}$ are not nugatory (Lemma~\ref{deference_fc}).
Hence $P_{j-1}\cap D$ is not affected by the sequence.
Then, we apply the sequence of deformations of type $\ri^{-}$ to $P_j$ ($\subset S^2$).
Let $\widetilde{P_{j}}$ be the spherical curve obtained from $P_j$ by applying the sequence.
Moreover let $S_{\iii}$ be a spherical curve obtained from $\reduced(P_{j-1})$ by applying a deformation of type $\riii$ to $\reduced(P_{j-1})\cap D$.
By the assumption: $f_{c}(P_{j}) \geq f_{c}(P_{j-1})$, we apply Claim~\ref{cycle} with regarding $Q=P_{j-1}, Q'=P_{j}, R=S_{\iii}$ and $\widetilde{P_{j}}=\widetilde{Q'}$, we have;

\begin{claim}\label{claim5-3-2}
$\widetilde{P_{j}}$ and $S_{\iii}$ are ambient isotopic.
\end{claim}

By the assumption of Case~3-1, $g(I_{\lambda})$ ($\lambda=1, 2$) is a simple arc.
Let $F$ be a sufficiently small disk completely including $g(I_1)$, $g(I_2)$ and the $3$-gon of $\reduced(P_{j-1})\cap D$.
We note that the configuration ($F$, $\reduced(P_{j-1})\cap F$) is as in Figure~\ref{configurationF,Pj-1capF}.
Then this figure shows that $\reduced(P_{j-1})$ admits a connected sum decomposition $\reduced(P_{j-1})^{\bigstar}\sharp 3_{1}$, 
where $\bigstar$ denotes a certain decomposition.
Here, we note that $\reduced(P_{j-1})^{\bigstar}$ may not be $\ri$-minimal.
For example, we obtained $\reduced(P_{j-1})$ as in the left figure of Figure~\ref{exampleConfigurationF}.
Then the spherical curve admits a connected sum decomposition as in the right figure of Figure~\ref{exampleConfigurationF}.
However, since $\reduced(P_{j-1})$ is $\ri$-minimal, we see that $\reduced(P_{j-1})^{\bigstar}\sharp 3_{1}$ may be further decompose into $\reduced(P_{j-1})^{\bigstar\bigstar} \sharp \infty^{m} \sharp 3_{1}$, where $\reduced(P_{j-1})^{\bigstar\bigstar}$ is $\ri$-minimal.
These show that $\reduced(P_{j-1})$ is obtained from $\reduced(P_{j-1})^{\bigstar\bigstar}$ by deformation of type $\beta^{+}$.

On the other hand, since $S_{\iii}$ is obtained $\reduced(P_{j-1})$ by deformation of type $\riii$, it is directly obtained from Figure~\ref{configurationF,Pj-1capF}, that $\reduced(S_{\iii})$ and $\reduced(P_{j-1})^{\bigstar\bigstar}$ are ambient isotopic.

These facts together with Claim~\ref{claim5-3-2} shows that $\reduced(\widetilde{P_{j}})$ is obtained from $\reduced(P_{j-1})$ by a deformation of type $\beta^{-}$.
Hence $\reduced(P')$ is obtained from $\reduced(P)$ by a deformation of type $\beta^{-}$.


\begin{figure}[htbp]
\centering
\includegraphics[width=6cm]{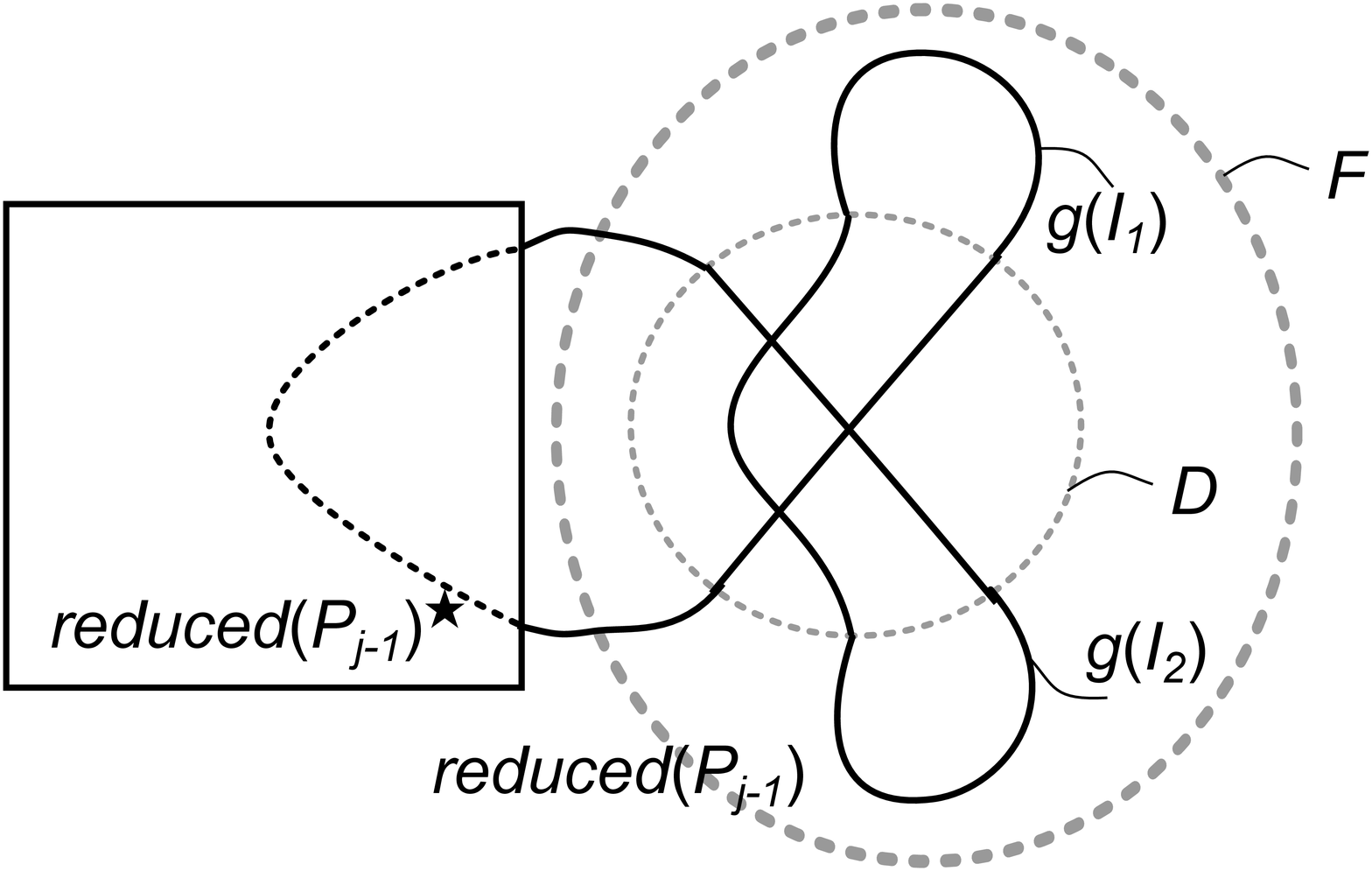}
\caption{}
\label{configurationF,Pj-1capF}
\end{figure}

\begin{figure}[htbp]
\centering
\includegraphics[width=13cm]{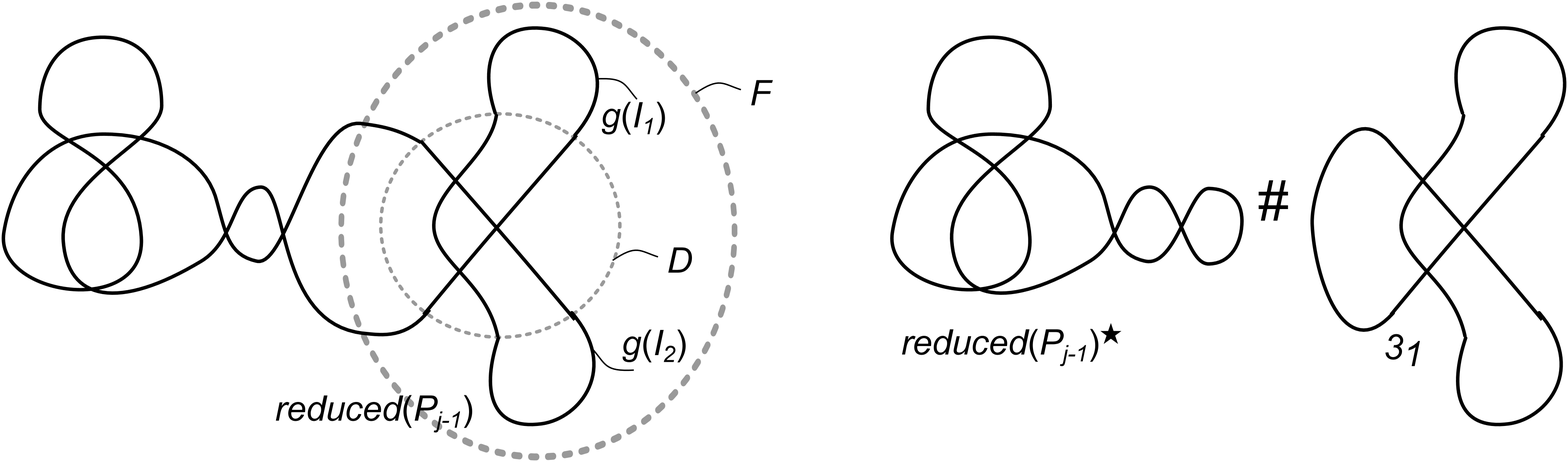}
\caption{}
\label{exampleConfigurationF}
\end{figure}

\medskip
\noindent $\bullet$ Case~3-2: 
$g(I_{\lambda})$ is a simple arc each of the others is not a simple arc.
Without loss of generality, we may suppose that $g(I_{1})$, $g(I_{2})$ are not simple arcs by and $g(I_3)$ is a simple arc.

In this case, it is easy to see that the arguments in Case~2-1 works to show that $\reduced(P')$ is obtained from $\reduced(P)$ by a single deformation of type ${\alpha}^{-}$.
Details are left to the reader.


\medskip
\noindent $\bullet$ Case~3-3: Every $g(I_{\lambda})$ ($\lambda = 1, 2, 3$) is not a simple arc.

The treatment of this case is essentially the same as the arguments in Case~1.
The difference is;

In Case~1, we appealed the connectedness of $P_{j-1}\setminus D$ (i.e., $\delta=1$), for contradictions.
Here we should use the condition of Case~3-3 for the connectedness of $P_{j-1}\setminus D$ to derive contradictions.

Then we can show that $\reduced(P')$ is obtained from $\reduced(P)$ by deformation of type $\riii$.

For every case, the statement is proved, which completes the proof.

$\hfill\Box$

\section*{Acknowledgements}
The authors would like to thank Professor Tsuyoshi Kobayashi for giving many advice.
The authors also thank Professor Kouki Taniyama for comments. 
The authors also thank Mr.~Yusuke Takimura for sharing his table and comments.
M.~H. is a researcher supported by Meiji University Organization for the Strategic Coordination of Research
 and Intellectual Properties.

\end{document}